\documentclass[12pt,reqno]{amsart}
\usepackage[utf8]{inputenc}
\usepackage[T1]{fontenc}
\usepackage[usenames, dvipsnames]{color}
\usepackage{ulem}

\usepackage{dsfont, amsfonts, amsmath, amssymb,amscd, stmaryrd, latexsym, amsthm, dsfont}
\usepackage[frenchb,english]{babel}
\usepackage{enumerate}
\usepackage{geometry}
\usepackage{float}
\usepackage{xypic}

\newtheorem{theorem}{Theorem}[section]
\newtheorem{lemma}[theorem]{Lemma}
\newtheorem{proposition}[theorem]{Proposition}
\newtheorem{corollary}[theorem]{Corollary}

\theoremstyle{remark}
\newtheorem{remark}[theorem]{\bf Remark}

\newtheorem{example}[theorem]{\bf Example}

\usepackage{hyperref}
\hypersetup{
	colorlinks=true,
	urlcolor=blue,
	citecolor=blue}
\def\intR{\mathrm{Int^R}}
\def\int{\mathrm{Int}}
\begin{document}
\selectlanguage{english}
\title[On the Krull dimension of $\intR(E,D)$]{On the Krull dimension of rings of integer-valued rational functions}

\author[M. M. Chems-Eddin]{M. M. Chems-Eddin}
\address{Mohamed Mahmoud CHEMS-EDDIN: Department of Mathematics, Faculty of Sciences Dhar El Mahraz, Sidi Mohamed Ben Abdellah University, Fez,  Morocco}
\email{2m.chemseddin@gmail.com}
 
\author[B. Feryouch]{B. Feryouch}
\address{Badr Feryouch: Department of Mathematics, Faculty of Sciences Dhar El Mahraz, Sidi Mohamed Ben Abdellah University, Fez,  Morocco}
\email{badr.feryouch@usmba.ac.ma}

\author[H. Mouanis]{H.  Mouanis}
\address{Hakima  Mouanis: Department of Mathematics, Faculty of Sciences Dhar El Mahraz, Sidi Mohamed Ben Abdellah University, Fez,  Morocco}
\email{hmouanis@yahoo.fr}

\author[A. Tamoussit]{A. Tamoussit}
\address{Ali TAMOUSSIT: Department of Mathematics, The Regional Center for Education and Training Professions Souss-Massa,  Inezgane, Morocco}
\email{a.tamoussit@crmefsm.ac.ma ; tamoussit2009@gmail.com}

	
\maketitle
\begin{abstract} 
Let $D$ be an integral domain with quotient field $K$ and $E$ a subset of $K$. The \textit{ring of integer-valued  rational functions on} $E$ is defined as $$\intR(E,D):=\lbrace \varphi \in K(X);\; \varphi(E)\subseteq D\rbrace.$$ The main goal of this paper is to investigate the Krull dimension of the ring $\intR(E,D).$ Particularly, we are interested in domains that are either Jaffard or PVDs. Interesting results are established with some illustrating examples.
\end{abstract}
	
\section*{Introduction}

The Krull dimension of a commutative ring $R$, denoted by $\dim(R)$, is defined as the supremum of the length of finite chains of prime ideals in $R$. Notably, the Krull dimension plays a pivotal role in both commutative algebra and algebraic geometry. Originally, this concept goes back to the works of Noether and Krull in the 1920s. Specifically, in 1937, Krull proposed the previous definition, which bears his name as a tribute to his contributions. This definition was justified by the accumulated evidence presented by Noether \cite{N23} for factor rings of polynomial rings and by R\"{u}ckert \cite{Ru32} for factor rings of power series rings. For a historical introduction to dimension theory, we refer the reader to  Chapter 8 of Eisenbud's book \cite{Eis95}. More recently, a constructive approach to studying the concept of Krull dimension of commutative rings has been developed in \cite[Chapter XIII]{LQ15}.

 \medskip

Given an integral domain $D$ with quotient field $K$. We denote by $\mathrm{Int}(D)$ the set of all polynomials with coefficients in $K$ that take values (on elements of $D$) in $D$, i.e., $$\mathrm{Int}(D)= \{f\in K[X];\; f(D)\subseteq D\}.$$
This set forms an integral domain lying between $D[X]$ and $K[X]$, and known as \textit{the ring of integer-valued polynomials over} $D$. More generally, the \textit{ring of integer–valued polynomials on a subset} $E$ of $K$ is defined by $$\mathrm{Int}(E,D)= \{f\in K[X];\; f(E)\subseteq D\}.$$ Obviously, $D\subseteq\mathrm{Int}(E,D)\subseteq K[X]$ and $\mathrm{Int}(D,D)=\mathrm{Int}(D).$

\medskip

In the definition of $\mathrm{Int}(E,D)$, if we replace $K[X]$ with $K(X)$, we obtain a new class of rings denoted by $\intR(E,D)$, which is called the \textit{ring of integer-valued rational functions on} $E$ \textit{over} $D,$ that is, $$\intR(E,D):=\{f\in K(X);\; f(E)\subseteq D\}.$$
 It is clear that  $\intR(E,D)$ is a $D$-algebra, $\int(E,D)\subseteq\intR(E,D)\subseteq K(X),$ $\intR(E,D)\cap K[X]=\int(E,D)$ and $\intR(E,D)\cap K=D$. We refer the reader to Chapter X of the book \cite{C97} for some background on these rings.
 
\medskip

At this stage, it is important to note that $\intR(D)$ and $\int(D)$ coincide for a specific class of integral domains known as $d$-rings, see for example   \cite{GM76,Mc78}. More precisely, an integral domain $D$ is defined as a $d$-\textit{ring} if, for any two nonzero polynomials $f$ and $g$ with coefficients in $D,$ such that $g(a)$ divides $f(a)$ for almost all elements $a$ in $D$, then $g$ divides $f$ in $K[X],$ where $K$ denotes the quotient field of $D.$ This definition is, in fact, equivalent to the equality $\intR(D)=\int(D)$ (see \cite[Definition VII.2.1]{C97} and the subsequent comment that follows it). The concept of $d$-rings was originally introduced and studied by Gunji and McQuillan in \cite{GM76}, notably, they proved that the ring of integers $\mathbb{Z}$ and the polynomial ring in any number of indeterminates over an integral domain are $d$-rings. They also pointed out that any integral domain with a nonzero Jacobson radical, including semi-local domains, is not a $d$-ring.

\medskip

Before continuing, it is convenient to recall some additional  concepts. The \textit{valuative dimension} of an integral domain $D$, denoted by $\dim_v(D)$, is defined to be the supremum of the Krull dimensions of the valuation overrings of $D$. Notice that  $\dim(D)\leqslant\dim_v(D),$ and when these two dimensions are equal and finite $D$ is called a \textit{Jaffard domain}. In the finite Krull dimensional case, Noetherian domains and Pr\"ufer domains are examples of Jaffard domains. It is known that for any integral domain $D$ and any overring $T$ of an integral domain $D,$ we have $\dim_v(D[X])=1+\dim_v(D)$ and $\dim_v(T)\leqslant\dim_v(D).$ For further details on valuative dimension and Jaffard domains, we refer the reader to \cite[Chapitre IV]{J60} and \cite{ABDFK88}, respectively.

\medskip

In \cite{S53}, Seidenberg established that for any integral domain $D$, the inequality $1 + \dim(D) \leqslant \dim(D[X]) \leqslant 1 + 2\dim(D)$ holds. Furthermore, he showed that $\dim(D[X]) = 1 + \dim(D)$ when $D$ is a Noetherian domain. This equality extends to a more general class of integral domains, specifically Jaffard domains (see, for example, \cite[Remark 1.3\rm(c)]{ABDFK88}).

\medskip

We now present some well-known results related to rings of integer-valued polynomials. In \cite[Propositions V.1.5 and V.1.6]{C97}, Cahen and Chabert established that for any integral domain $D,$ we have $\dim(\int(D))\geqslant\mathrm{max}\{1+\dim(D),\dim(D[X])-1\}.$ They also showed that for any fractional subset $E$ of a Jaffard domain $D,$ the Krull dimension of $\int(E,D)$ is equal to $1+\dim(D)$ (\cite[Proposition V.1.8]{C97}). Additionally, as asserted in \cite[Exercise V.5.(ii)]{C97}, the valuative dimension satisfies the equality $\dim_v(\int(E, D))=1+\dim_v(D)$ for any fractional subset $E$ of an arbitrary integral domain $D.$

\medskip

Tartarone \cite{T97} provided significant insights into the Krull dimension of $\int(D)$ and, more generally $\int(E,D),$  for integral domains $D$  arising from pullback diagrams, including (locally) pseudo-valuation domains. Notably, she proved that for any pseudo-valuation domain $D$ with finite residue field and for any fractional subset $E$ of $D,$ we have  $\dim(\mathrm{Int}(E,D))=1+\dim(D).$ 

\medskip

Later, in \cite[Theorem 2.1]{FK04}, Fontana and Kabbaj \cite{FK04} proved that $\dim(\mathrm{Int}(D))=\dim(D[X])$ for any locally essential domain $D$.

\medskip

More recently, in \cite{COT23}, the authors established that for certain Jaffard-like domains $D,$ and under the assumption that $E$ is residually cofinite with $D,$ the Krull dimension of any ring lying between $D[X]$ and $\mathrm{Int}(E,D)$ is equal to the Krull dimension of $D[X].$ This result indeed extends \cite[Theorem 2.1]{FK04}.

\medskip

Lastly, it is worth noting that rings of integer-valued rational functions have been studied less extensively than integer-valued polynomials, particularly regarding their Krull dimension. In fact, as far as we know, the only known result about the Krull dimension is that $\dim(\intR(D))\geqslant1+\dim(D),$ with equality if $D$ is a Jaffard domain; see \cite[Exercise X.1]{C97}.


\medskip	 	

The aim of this paper is to investigate the Krull dimension of the ring $\intR(E,D)$. Among other things, we prove that for any fractional subset $E$ of an integral domain $D$, we have that $D$ is Jaffard if and only if $\intR(E,D)$ is a Jaffard domain and the Krull dimension equal to  $1+\dim(D)$ (Theorem \ref{dk}). Moreover, we show that $1+\dim(D)\leqslant\dim(\intR(E,D))\leqslant1+\dim_v(B)$ if $D\subset B$ is a pair of integral domains sharing a common non zero ideal $I$ such that $D/I$ is finite and $E$ is a fractional subset of $D$ contained in $B$ (Theorem \ref{Psv}). In particular, as an application of this last inequality, we deduce that the equality $\dim(\intR(E,D))=1+\dim(D)$ holds when $D$ is a pseudo-valuation domain with finite residue field and $E$ is a fractional subset of $D$ contained in the associated valuation domain of $D$ (Corollary \ref{CorPVD}).
	  
\medskip

Throughout this paper, $D$ is an integral domain of quotient field $K$ and $E$ is a nonempty subset of $K$. Moreover, the symbols $\subset$ and $\subseteq$ denote proper containment and large containment, respectively. 
 	  	 
\section{Main results and examples}
We start this section by recalling the following basic facts :

\begin{enumerate}
\item[\textbf{Fact 1:}]  If $D\subseteq B$ are two integral domains with the same quotient field $K$ and  $E$  a subset of $K$, then $\intR(E,D)\subseteq \intR(E,B)$.
\item[\textbf{Fact 2:}]  $\intR(E,D)$ is an overring of $D[X]$ if and only if $E$ is contained in $D$.
\item[\textbf{Fact 3:}]  For any prime ideal $\mathfrak{p}$ of $D$ and for any element $a$ of $E$,   $\mathfrak{P}_{\mathfrak{p},a}:=\lbrace \varphi \in \intR(E,D);\; \varphi(a) \in \mathfrak{p} \rbrace$  is a prime ideal of $\intR(E,D)$ above $\mathfrak{p}.$
\end{enumerate}

\medskip

Recall that a subset $E$ of $K$ is said to be a \textit{fractional subset} of $D$ if  there exists a nonzero element $d$ of $D$ such that $dE\subseteq D.$\\
The next result gives some bounds about the Krull dimension of $\intR(E,D)$.
\begin{proposition}\label{kd}
We have $\dim(\intR(E,D))\geqslant \dim(D)$, and if moreover $E$ is a fractional subset of $D$ then $\dim(\intR(E,D))\geqslant 1 + \dim(D)$.
\end{proposition}

\begin{proof}
Set $n=\dim(D),$ let $a$ be an element of $E$ and let $(0) = \mathfrak{p}_{0}\subset \mathfrak{p}_{1}\subset \dots\subset \mathfrak{p}_{n}$ be a chain of prime ideals of length $n$ in $D.$ We first note that the inclusion $\mathfrak{p}_{i}\subset \mathfrak{p}_{i+1}$ implies a strict inclusion $\mathfrak{P}_{\mathfrak{p}_{i},a} \subset \mathfrak{P}_{\mathfrak{p}_{i+1},a}$ in $\intR(E,D).$ Indeed, if not, the fact that the prime ideal $\mathfrak{P}_{\mathfrak{p}_{i},a}$ (resp., $\mathfrak{P}_{\mathfrak{p}_{i+1},a}$) is above $\mathfrak{p}_{i}$ (resp., $\mathfrak{p}_{i+1}$) implies the equality $\mathfrak{p}_{i}=\mathfrak{p}_{i+1},$ which is a contradiction. Consequently, $\mathfrak{P}_{\mathfrak{p}_{0},a} \subset \mathfrak{P}_{\mathfrak{p}_{1},a} \subset \mathfrak{P}_{\mathfrak{p}_{2},a} \subset \dots \subset \mathfrak{P}_{\mathfrak{p}_{n},a}$ is a chain of prime ideals of length $n$ in $\intR(E,D),$ and thus $\dim(\intR(E,D))\geqslant\dim(D).$ For the moreover statement, assume that $E$ is a fractional subset of $D$. By definition, there is a nonzero element  $d$ of $D$ such that $dE\subseteq D$, and then $\mathfrak{P}_{\mathfrak{p}_{0},a}\neq (0)$ because the polynomial $dX-da$  lies in $\mathfrak{P}_{\mathfrak{p}_{0},a}.$ Hence, $(0)\subset \mathfrak{P}_{\mathfrak{p}_{0},a} \subset \mathfrak{P}_{\mathfrak{p}_{1},a} \subset \mathfrak{P}_{\mathfrak{p}_{2},a} \subset \dots \subset \mathfrak{P}_{\mathfrak{p}_{n},a}$ forms a chain of prime ideals of length $n+1$ in $\intR(E,D)$. Therefore, $\dim(\intR(E,D))\geqslant 1 + \dim(D),$ as desired.
\end{proof}

\begin{remark}
From the previous proposition, we deduce that if the dimension of $D$ is infinite, then so is the dimension of $\intR(E,D)$. However, we do not know whether the converse is true or not.
\end{remark}


It is well-known that $\dim_v(T)\leqslant\dim_v(D),$ for any overring $T$ of an integral domain $D$. So, from this fact, we deduce the following: 
\begin{proposition}
For any subset $E$ of $K$, the following inequality holds $$\dim_v(\mathrm{Int}(E,D))\geqslant \dim_v(\intR(E,D)).$$
\end{proposition}

In what follows, we provide the analogue of Proposition \ref{kd} for the valuative dimension.

\begin{proposition}\label{valudim}
We have $\dim_v(\intR(E,D))\geqslant \dim_v(D)$, and if moreover $E$ is a fractional subset of $D$ then $\dim_v(\intR(E,D))= 1 + \dim_v(D)$.
\end{proposition}

\begin{proof}
Let $V$ be a valuation overring of $D$ such that $\dim_v(D)=\dim(V)$. It follows from Proposition \ref{kd} that $\dim(V)\leqslant \dim(\intR(E,V))$, and since $\intR(E,D)\subseteq\intR(E,V)$, this implies 
  $$\dim_{v}(D)=\dim(V)\leqslant \dim(\intR(E,V))\leqslant \dim_{v}(\intR(E,V))\leqslant \dim_{v}(\intR(E,D)).$$
For the moreover statement, assume that $E$ is a fractional subset of $D$. We claim first that $\dim_{v}(\intR(E,V))=1+\dim(V)$. As a matter of fact, since $E$ is also a fractional subset of $V$ and $\mathrm{Int}(E,V)\subseteq\intR(E,V)$, we have   $$1+\dim(V)\leqslant\dim(\intR(E,V))\leqslant\dim_{v}(\intR(E,V))\leqslant\dim_v(\mathrm{Int}(E,V))=1+\dim(V)$$ (the last equality follows from \cite[Proposition V.1.8]{C97}). So, using the previous claim and from the inclusion  $\intR(E,D)\subseteq\intR(E,V)$, we deduce that $$1+\dim_v(D)=1+\dim(V)=\dim_{v}(\intR(E,V))\leqslant\dim_{v}(\intR(E,D)).$$
For the other inequality, we have $\dim_v(\intR(E,D)) \leqslant \dim_v(\int(E,D))$ because  $\int(E,D)\subseteq \intR(E,D)$. So, by \cite[Exercise V.5]{C97}, $\dim_v(\int(E,D)) = 1+\dim_v(D)$, and hence  $\dim_v(\intR(E,D)) \leqslant  1+\dim_v(D)$. Therefore,  $\dim_v(\intR(E,D))=1+\dim_v(D)$.
\end{proof}

From this last result, we derive the following two corollaries.
\begin{corollary}\label{ega}
For any fractional subset $E$ of a Jaffard domain  $D$, we have $$\dim(\intR(E,D))=1+\dim(D).$$
\end{corollary}

\begin{proof}
Let $D$ be a Jaffard and $E$ is a fractional subset of $D$. Since $D$ is Jaffard and $\dim(\intR(E,D))\leqslant \dim_v(\intR(E,D))$, it follows from Proposition \ref{valudim} that $\dim(\intR(E,D))\leqslant 1 + \dim_v(D)=1 + \dim(D)$. Moreover, the other inequality is verified in Proposition \ref{kd}. Thus, $\dim(\intR(E,D))=1 + \dim(D)$, as wanted.  
\end{proof}

Notice that, according to \cite[Proposition 1.2\rm(a)]{ABDFK88}, if $D$ is a Jaffard domain, then $D[X]$ is also a Jaffard domain. However, it follows from \cite[Example 3.6]{FK90} that  $D[X]$ can be a Jaffard domain whereas $D$ is not. Lastly, as pointed out in \cite[Remark 1.3(c)]{ABDFK88}, $D$ is a Jaffard domain if and only if $D[X]$ is a Jaffard domain with Krull dimension is equal  to $1+\dim(D).$

\begin{corollary}\label{CoDX}
If $D[X]$ is Jaffard and $E$ is a fractional subset of  $D$, then  $$\dim(\intR(E,D))\leqslant\dim(D[X]).$$
\end{corollary}

\begin{proof}
Assume that $D[X]$ is Jaffard and $E$ a fractional subset of $D$. We have $\dim(D[X]) = \dim_v(D[X])$, and so by Proposition \ref{valudim},  $$\dim_v(\intR(E,D)) = 1 + \dim_v(D) = \dim_v(D[X]) = \dim(D[X]),$$where the second equality follows from \cite[Théorème 2 du Chapitre IV]{J60}. Therefore, $\dim(\intR(E,D)) \leqslant \dim(D[X]).$
\end{proof}   

With a slight modification of the proofs of Proposition \ref{valudim} and Corollary \ref{CoDX}, we easily obtain the following:
\begin{proposition}\label{ProDimV}
Let $E$ be a subset of $D$ and let $B$ be a ring between $D[X]$ and $\intR(E,D)$. We have the following statements:
\begin{enumerate}[$(1)$]  
\item  $\dim_v(B)=\dim_v(D[X])$. 
\item If $D[X]$ is Jaffard, then $\dim(B)\leqslant\dim(D[X])$.
\end{enumerate}
\end{proposition}
If $\mathfrak{p}$ is a prime ideal of $D$, we set $\intR(D)_\mathfrak{p}:=\intR(D)_{(D\setminus\mathfrak{p})}.$
In the next result, we show that $\intR(D)_\mathfrak{p}$ and $\intR(D_\mathfrak{p})$ have the same valuative dimension for all prime ideals $\mathfrak{p}$ of $D$.
\begin{lemma}\label{LemLocalDim}
We have $$\dim(\intR(D))=\sup\left\{\dim(\intR(D)_\mathfrak{m});\,\mathfrak{m}\in\mathrm{Max}(D)\right\}.$$ 
\end{lemma}
\begin{proof}
This is a particular case of \cite[Lemma 1.2]{Tam21}.
\end{proof}
\begin{proposition}\label{v_dimLo}
For each prime ideal $\mathfrak{p}$ of $D,$ we have $$\dim_v(\intR(D)_\mathfrak{p})=\dim_v(\intR(D_\mathfrak{p})) = 1 +\dim_v(D_\mathfrak{p}).$$
\end{proposition}
\begin{proof}
Let $\mathfrak{p}$ be a prime ideal of $D$. The second equality follows directly from Proposition \ref{valudim}. To prove that $\dim_v(\intR(D)_\mathfrak{p})= 1 +\dim_v(D_\mathfrak{p}),$ we need first to check the inclusion $\intR(D)_\mathfrak{p}\subseteq\intR(D,D_\mathfrak{p})$. To do this, let $\varphi \in \intR(D)_\mathfrak{p}$. Then there exists an element $x$ of $D\backslash\mathfrak{p}$ such that $x\varphi \in \intR(D)$, and hence $x\varphi\in K(X)$ and $x\varphi(d)\in D$, for all $d\in D$. Thus, $\varphi\in K(X)$ and $\varphi(d)\in D_\mathfrak{p}$ for all $d\in D$, that is, $\varphi \in \intR(D,D_\mathfrak{p}).$ Thus, we have $D_\mathfrak{p}[X]\subseteq\intR(D)_\mathfrak{p}\subseteq\intR(D,D_\mathfrak{p}),$ and therefore $$\dim_v(\intR(D,D_\mathfrak{p}))\leqslant \dim_v(\intR(D)_\mathfrak{p}))\leqslant\dim_v(D_\mathfrak{p}[X])=1+\dim_v(D_\mathfrak{p}).$$ 
Furthermore, by Proposition \ref{valudim}, $\dim_v(\intR(D,D_\mathfrak{p}))=1+\dim_v(D_\mathfrak{p}),$ and so we deduce the equality $\dim_v(\intR(D)_\mathfrak{p})=1+\dim_v(D_\mathfrak{p}).$ 
\end{proof}
Proposition \ref{v_dimLo} offers us an analogous result to Lemma \ref{LemLocalDim} concerning the valuative dimension.
\begin{corollary}
We have 
\begin{align*}
\dim_v(\intR(D))& =\sup\left\{\dim_v(\intR(D)_{\mathfrak{m}});\,\mathfrak{m}\in\mathrm{Max}(D)\right\}\\
&=\sup\left\{\dim_v(\intR(D_{\mathfrak{m}}));\,\mathfrak{m}\in\mathrm{Max}(D)\right\}
\end{align*}
\end{corollary}
\begin{proof}
We first notice that for any integral domain $R$, we have $$\dim_v(R)=\sup\left\{\dim_v(R_\mathfrak{m});\,\mathfrak{m}\in\mathrm{Max}(R)\right\}$$
(this equality holds because any valuation overring $V$ of $R$  always contains $R_\mathfrak{m},$ where $\mathfrak{m}$ is a maximal ideal containing the center of $V$ on $R$).\\
From Proposition \ref{valudim}, we have $\dim_v(\intR(D))=1+\dim_v(D),$ and then 
\begin{align*}
\dim_v(\intR(D))& =1+\sup\left\{\dim_v(D_\mathfrak{m});\,\mathfrak{m}\in\mathrm{Max}(D)\right\}\\
&=\sup\left\{1+\dim_v(D_\mathfrak{m});\,\mathfrak{m}\in\mathrm{Max}(D)\right\}\\
& =\sup\left\{\dim_v(\intR(D)_{\mathfrak{m}});\,\mathfrak{m}\in\mathrm{Max}(D)\right\}\\
&=\sup\left\{\dim_v(\intR(D_{\mathfrak{m}}));\,\mathfrak{m}\in\mathrm{Max}(D)\right\}
\end{align*}
(these last two equalities are due to Proposition \ref{v_dimLo}).
\end{proof}
\begin{corollary}
We have $$ 1+\dim(D)\leqslant\dim(\intR(D))\leqslant 1 +\dim_v(D).$$
\end{corollary}
\begin{proof}
The left inequality follows from  Proposition \ref{kd}. For the right inequality, we use the fact that $\dim_v(D)=\sup\left\{\dim_v(D_\mathfrak{m});\,\mathfrak{m}\in\mathrm{Max}(D)\right\}$, and we combine this with  Lemma \ref{LemLocalDim} and Proposition \ref{v_dimLo}.
\end{proof}
\begin{remark}
For any Jaffard domain $D$, the equality $\dim(\intR(D))=1 +\dim(D)$ follows directly from Corollary \ref{ega}. It is worth noting that the same equality is obviously derived from the previous corollary.
\end{remark}

Now, we present our first main result of this paper.
\begin{theorem}\label{dk}
Assume that $E$ is a fractional subset of $D$. Then, $D$ is Jaffard if and only if $\intR(E,D)$ is Jaffard and $\dim(\intR(E,D)) = 1 +\dim(D)$.
\end{theorem}

\begin{proof} 
If $D$ is Jaffard, then it follows from Proposition \ref{valudim} and Corollary \ref{ega} that   $\dim_v(\intR(E,D)) = 1 + \dim_v(D) = 1 + \dim(D) = \dim(\intR(E,D)).$ Thus, $\intR(E,D)$ is Jaffard and $\dim(\intR(E,D))=1+ \dim(D)$. For the reverse implication, assume that $\intR(E,D)$ is Jaffard and $\dim(\intR(E,D)) = 1 +\dim(D)$. Then,  $\dim_v(\intR(E,D)) = \dim(\intR(E,D))= 1 +\dim(D)$, and hence   $1 + \dim_v(D) = 1 + \dim(D)$ (by Proposition \ref{valudim}). Therefore, $\dim(D)=\dim_v(D),$ that is, $D$ is Jaffard.
\end{proof}

As illustrative examples, consider the following:
\begin{example}\label{Exa1}
The ring $D = \mathbb{Z} + T\mathbb{Q}[T]$, where $T$ is an indeterminate over $\mathbb{Q}$, is known to be a two-dimensional Pr\"ufer domain. According to Theorem \ref{dk}, $\intR(E,D)$ is a three-dimensional Jaffard domain, for any fractional subset $E$ of $D.$
\end{example}

\begin{example}\label{Exa2}
Let $S$ be the multiplicative subset of $\mathbb{Z}$ generated by the set of primes consisting of $2$ and all odd primes $p$ such that of prime numbers $p\equiv 1(\mod 4),$ and set $D=S^{-1}\mathbb{Z}.$ As $D$ is a Dedekind domain, it follows from Theorem \ref{dk} that $\intR(E,D)$ is a two-dimensional Jaffard domain, for any fractional subset $E$ of $D.$
\end{example}

\begin{remark}
\rm(1)  Since the integral domain $D$ in Example \ref{Exa1} lies between $\mathbb{Z}[T]$ and $\mathbb{Q}[T]$, it follows from \cite[Lemma VII.2.10]{C97} that $D$ is a $d$-ring, and thus  $\intR(D)$ coincides with the classical ring  $\int(D).$ However, it remains an open question whether $\intR(E,D)\neq \int(D)$ for some subsets $E$ of $D.$

\rm(2) In \cite[Example 1.11]{Lo88}, the author shows that the integral domain $D$ presented in Example \ref{Exa2} is not a $d$-ring. Thus, for any subset $E$ of $D$, we have $\intR(E,D)\neq\int(D)$, as $\intR(D)$ is always a subring of $\intR(E,D)$.  
\end{remark}

In what follows, we focus on the Krull dimension of $\intR(E,D)$ when $D$ is a pullback and $E$ a subset of the quotient field of $D.$ Specifically, for certain pairs of integral domains $D\subset B$ and subsets $E$ of the quotient field of $D,$ we establish a bound on $\dim(\intR(E,D))$ in terms of $\dim(D)$ and $\dim_v(B).$  The obtained results are particularly useful for studying pseudo-valuation domains, as some of them are classical examples of non-Jaffard domains. 

\begin{theorem}\label{Psv}
Let $D\subset B$ be a pair of integral domains sharing a common nonzero ideal $I$ and $E$ a nonempty subset of $K,$ where $K$ is the quotient field of both $D$ and $B$. If $D/I$ is finite, then $$\dim(\intR(E,D))\leqslant \dim(\intR(E,B)).$$ If, moreover, $E$ is a fractional subset of $D$ contained in $B$, then $$1+\dim(D)\leqslant\dim(\intR(E,D))\leqslant 1+\dim_v(B).$$
\end{theorem}
The proof of this theorem requires the following preparatory lemmas.
\begin{lemma}\label{Shr}
Let $R\subset S$ be a pair of integral domains sharing a common nonzero ideal $I$. Then, $\dim(R)\leqslant\dim(S)+\dim(R/I).$
\end{lemma}
\begin{proof}
See \cite[Corollaire 1 du Théorème 1, page 509]{C88}.
\end{proof}

For any ideal $I$ of $D$, we set $\intR(E,I):=\{\varphi\in K(X);\,\varphi(E)\subseteq I\}.$

\begin{lemma}\label{Cid}
Let $I$ be an ideal of $D$ and $E$ a nonempty subset of $K$. We have
\begin{enumerate}[$(1)$]  
\item  $\intR(E,I)$ is an ideal of $\intR(E,D)$.
\item  If $D/I$ is finite, then the ring $\intR(E,D)/\intR(E,I)$ is zero-dimensional.
\end{enumerate} 
\end{lemma}
\begin{proof}
(1) This is clear.
(2) Assume that $D/I$ is finite with cardinality $r.$ To prove that the ring $\intR(E,D)/\intR(E,I)$ is zero-dimensional,  it suffices to show that any prime ideal of $\intR(E,D)$ containing $\intR(E,I)$ is maximal (a fact apperead in \cite{C97} as an exercise). So, let $\mathfrak{P}$ be such a prime ideal, and consider  a set of representatives $\lbrace u_1,...,u_r\rbrace$ of $D$ modulo $I.$ Arguing as in the proof of \cite[Lemma V.1.9]{C97}, we take an element $\varphi$ in $\intR(E,D).$ Then, we have that the product $\prod_{i=1}^{r}(\varphi(X)-u_i)$ lies in $\intR(E,I)$, and hence in $\mathfrak{P}$. Thus one of its factors must also lie in $\mathfrak{P}$, and therefore $\lbrace u_1,...,u_r\rbrace$ is a set of representatives of $\intR(E,D)$ modulo $\mathfrak{P}$. Consequently, $\intR(E,D)/\mathfrak{P}$ is a field, which is equivalent to saying that $\mathfrak{P}$ is maximal, and so we deduce that $\dim(\intR(E,D)/\intR(E,I))=0$.
\end{proof}

\begin{proof}[Proof of Theorem \ref{Psv}]
Since $\intR(E,D)\subset\intR(E,B)$ and $\intR(E,I)$ is a common ideal of $\intR(E,D)$ and $\intR(E,B),$ we consider the following pullback diagram:
$$\xymatrix{
    \intR(E,D) \ar[r] \ar[d]  & \intR(E,D)/\intR(E,I) \ar[d] \\
    \intR(E,B) \ar[r] & \intR(E,B)/\intR(E,I),
  }$$ and so it follows from Lemma \ref{Shr}  that $$\dim(\intR(E,D))\leqslant\dim(\intR(E,B))+\dim(\intR(E,D)/\intR(E,I)).$$
By Lemma \ref{Cid},  $\dim(\intR(E,D)/\intR(E,I))=0,$ and hence, $\dim(\intR(E,D))\leqslant \dim(\intR(E,B)),$ as wanted.\\
We now prove the moreover statement. It is clear that the left inequality follows from Proposition \ref{kd}. For the right inequality, from the inclusion  $B[X]\subseteq\intR(E,B)$ we deduce that $\dim_v(\intR(E,B))\leqslant\dim_v(B[X]),$ and so $\dim(\intR(E,B))\leqslant1+\dim_v(B).$ Since $\dim(\intR(E,D))\leqslant \dim(\intR(E,B))$ (as previously established), it follows that $\dim(\intR(E,D))\leqslant1+\dim_v(B).$
\end{proof}
As an immediate application of Theorem \ref{Psv}, we derive the following:
\begin{corollary}
Let $D\subset B$ be a pair of integral domains sharing a common nonzero ideal $I$ and $E$ a fractional subset of $D$ contained  in $B$. If $D/I$ is finite and $B$ is Jaffard, then $$1+\dim(D)\leqslant\dim(\intR(E,D))\leqslant 1+\dim(B).$$
\end{corollary}

Following \cite{HH78}, an integral domain $D$ is called a \textit{pseudo-valuation domain} (in short, a PVD) if every prime ideal $\mathfrak{p}$ has the property that whenever a product of two elements of the quotient field of $D$ lies in $\mathfrak{p},$ then one of the given elements is in $\mathfrak{p}$. In \cite{HH78}, the authors showed that valuation domains form a proper subclass of PVDs.  Furthermore, they established that a PVD $D$ that is not valuation must be a local domain sharing its maximal ideal $\mathfrak{m}$ with a valuation overring $V$.

\begin{corollary}\label{CorPVD}
Let $(D,\mathfrak{m})$ be a PVD with associated valuation overring $V$ and $E$ a fractional subset of $D$ contained in $V.$ If the residue field of $D$ is finite, then $$\dim(\intR(E,D))=1+\dim(D).$$ 
\end{corollary}
\begin{proof}
The desired equality follows by applying the previous corollary with $B=V$ and $I=\mathfrak{m}$, and by using the well-known fact that $D$ and $V$ have the same Krull dimension. 
\end{proof}
\begin{remark}
From Proposition \ref{valudim} and Corollary \ref{CorPVD}, we can determine whether the Jaffard property holds for $\intR(E,D)$ when $D$ is a PVD with finite residue field (for instance, see Example \ref{ExJaf}). In fact, in this case, it follows from Theorem \ref{dk} that $\intR(E,D)$ is Jaffard if and only if so is $D$.
\end{remark}
Here are two illustrative examples.

\begin{example}
Let $k$ be a finite field and $L$ a field containing $k.$ Consider the domain $D:=k+TL[[T]],$ where $T$ is an indeterminate over $L.$ This integral domain is known to be a one-dimensional PVD with a finite residue field, and, in addition, it is Jaffard if and only if the field extension $L/k$ is algebraic (\cite[Proposition 2.5\rm(b)]{ABDFK88}). Therefore, by applying Corollary \ref{CorPVD}, the Krull dimension of $\intR(E,D)$ is equal to two, for any subset $E$ of $D.$ 
\end{example}

\begin{example}\label{ExJaf}
 Let  $Y$ and $Z$ be two indeterminates over a finite field $k$ and set $D=k+Zk(Y)[Z]_{(Z)}$. Let $E$ be a fractional subset of $D$. It is known that $D$ is a one-dimensional PVD that is not Jaffard because $k[Y]_{(Y)}+Zk(Y)[Z]_{(Z)}$ is a two-dimensional valuation overring of $D.$ However, as noticed in \cite[Example 3.6]{FK90}, $D[X]$ is a three-dimensional Jaffard domain. By Corollary \ref{CorPVD}, we have   $\dim(\intR(E,D))=1+\dim(D)=2.$ Moreover, it follows from Proposition \ref{valudim} that $\dim_v(\intR(E,D))=1+\dim_v(D)=3,$ and thus $\intR(E,D)$ is not Jaffard. 
\end{example}
\begin{remark}
\rm(1) As mentioned in the introduction, a semi-local domain cannot be a $d$-ring. Consequently, in the two previous examples, $\intR(D),$  and hence $\intR(E,D),$ differs from the classical ring $\int(D).$



\rm(2) We point out that Example \ref{ExJaf} shows also that the inequality in Corollary \ref{CoDX} may be strict, that is, $\dim(\intR(E,D)) < \dim(D[X])$.
\end{remark}

We now provide a result on the Krull dimension of $\intR(E,D)$ when $D$ has  finite character. This is in fact the analogue of \cite[Proposition 1.7]{T97}. 

An integral domain $D$ is said to be of \textit{finite character}, if every nonzero element of $D$ belongs to only finitely many maximal ideals of $D$. Clearly, semi-local domains and Dedekind domains are examples of domains with finite character.

\begin{lemma}[{\cite[Lemma 1.6]{T97}}]\label{tlf}
Assume that $D$ has finite character and for each maximal ideal $\mathfrak{m}$ of $D$, $D_\mathfrak{m}$ shares the maximal ideal $\mathfrak{m}D_\mathfrak{m}$ with an overring $W_\mathfrak{m}$, such that $\mathfrak{m}D_\mathfrak{m}$ is a prime ideal in $W_\mathfrak{m}$. We fix a maximal ideal $\mathfrak{m}_{0}$ of $D$ and we set $C=C(\mathfrak{m}_{0})=W_{\mathfrak{m}_{0}}\cap (\cap_{\mathfrak{m}\neq \mathfrak{m}_{0}}D_\mathfrak{m})$. Then, $\mathfrak{m}_{0}$ is a prime ideal in $C$ and $C_{\mathfrak{m}_{0}}=W_{\mathfrak{m}_{0}}.$
\end{lemma}

\begin{proposition}\label{lpsvlf}
Under the notation and assumptions of Lemma \ref{tlf}, if $D$ has finite residue fields and $E$ is a subset of $D$, then $$1+\dim(D)\leqslant \dim(\intR(E,D))\leqslant \max\lbrace 1+\dim_v(W_{\mathfrak{m}});\;\mathfrak{m}\in \mathrm{Max}(D) \rbrace.$$
\end{proposition}

\begin{proof}
We need only check the right inequality because the left one follows from Proposition \ref{kd}. So, let $\mathfrak{m}$ be a maximal ideal of $D$. By assumptions, it follows from  Lemma \ref{tlf} that $C(\mathfrak{m})$ and $D$ share the ideal $\mathfrak{m}$, and then $\intR(E,D)$ and $\intR(E,C(\mathfrak{m}))$ share the ideal $\intR(E,\mathfrak{m})$.  Thus, $\intR(E,D)_{\mathfrak{m}}$ and $\intR(E,C(\mathfrak{m}))_{\mathfrak{m}}$ share the ideal $(\intR(E,\mathfrak{m}))_{\mathfrak{m}}$. Since $D/\mathfrak{m}$ is finite, we deduce from Theorem \ref{Psv} that $\dim(\intR(E,D)_{\mathfrak{m}})\leqslant \dim(\intR(E,C(\mathfrak{m}))_{\mathfrak{m}}).$ On the other hand, according to Lemma \ref{tlf}, we have $(C(\mathfrak{m}))_{\mathfrak{m}}[X]=W_{\mathfrak{m}}[X]\subseteq \intR(E,C(\mathfrak{m}))_{\mathfrak{m}}$ because $C(\mathfrak{m})[X]\subseteq \intR(E,C(\mathfrak{m})).$ Hence, $\dim(\intR(E,D)_{\mathfrak{m}})\leqslant \dim(\intR(E,C(\mathfrak{m}))_{\mathfrak{m}})\leqslant \dim_{v}(W_{\mathfrak{m}}[X])=1+\dim_{v}(W_{\mathfrak{m}}).$
Thus, the conclusion follows from Lemma \ref{LemLocalDim}.
\end{proof}
Recall that an integral domain $D$ is said to be \textit{locally PVD}, if $D_\mathfrak{m}$ is a PVD for all maximal ideals $\mathfrak{m}$ of $D.$ Note that while PVDs and Pr\"ufer domains are locally PVD, a locally PVD  needs not be a PVD or a Pr\"ufer domain, as illustrated in \cite[Example 2.5]{DF83}. Alternatively, the ring of integers $\mathbb{Z}$ is a locally PVD that is not a PVD, and any PVD that is not a valuation domain serves as an example of a locally PVD that is not Pr\"ufer. 

\begin{corollary}
Let $D$ be a locally PVD of finite character. If $D$ has finite residue fields and $E$ is a subset of $D$, then $$\dim(\intR(E,D))=1+\dim(D).$$
\end{corollary}

\begin{proof}
In this situation, $W_{\mathfrak{m}}$ is a valuation domain for all maximal ideals $\mathfrak{m}$ of $D$. Then, $\dim(W_{\mathfrak{m}})=\dim(D_{\mathfrak{m}}),$ and hence, by Proposition \ref{lpsvlf},  $\dim(\intR(E,D))\leqslant \max\left\{1+\dim(D_\mathfrak{m});\,\mathfrak{m}\in\mathrm{Max}(D)\right\}=1+\dim(D).$
Therefore, from Proposition \ref{kd}, we deduce that $\dim(\intR(E,D))=1+\dim(D).$
\end{proof}
 \begin{remark}
 \rm(1) We note that $\mathbb{Z}$ (or more generally, the ring of integers  $\mathbb{Z}_K$ of a number field $K$) is an example of a locally PVD with finite character, finite residue fields, and infinitely many maximal ideals. However, no example of a locally PVD with these properties has been found that is not a Jaffard domain.

 \rm(2) It is noteworthy that the aforementioned result extends  Corollary \ref{CorPVD}.
 \end{remark}

\bigskip

\textbf{Acknowledgements.} 
The authors are deeply grateful to the referee for providing valuable and detailed  comments that have greatly helped to improve this paper.

\end{document}